\documentclass[a4paper,12pt]{amsart}

\usepackage{amsfonts}
\usepackage{amsmath}
\usepackage{amssymb}

\usepackage[usenames]{color}
\usepackage[colorlinks]{hyperref}

\numberwithin{equation}{section}
\theoremstyle{plain}
\newtheorem{thm}{Theorem}[section]

\newtheorem{lemma}[thm]{Lemma}

\theoremstyle{definition}
\newtheorem{defi}[thm]{Definition}
\newtheorem{rem}[thm]{Remark}
\newtheorem{assumption}[thm]{Assumption}

\begin{document}

\title[Cauchy type problems for fractional differential equation]
{Initial, inner and inner-boundary problems for a fractional differential equation}

\author[Erkinjon Karimov]{Erkinjon Karimov}
\address{
  Erkinjon Karimov:
  \endgraf
  Institute of Mathematics, Uzbekistan Academy of Sciences
  \endgraf
  81 Mirzo Ulugbek street, Tashkent, 100170
  \endgraf
  Uzbekistan
  \endgraf
  {\it E-mail address} {\rm erkinjon@gmail.com}
  and
  \endgraf
 FracDiff Research Group
 \endgraf
Sultan Qaboos University
\endgraf
Oman
\endgraf
  {\it E-mail address} {\rm erkinjon@squ.edu.om}
 }

\author[Michael Ruzhansky]{Michael Ruzhansky}
\address{
  Michael Ruzhansky:
  \endgraf
Department of Mathematics: Analysis,
Logic and Discrete Mathematics
  \endgraf
Ghent University, Belgium
  \endgraf
 and
  \endgraf
 School of Mathematical Sciences
 \endgraf
Queen Mary University of London
\endgraf
United Kingdom
\endgraf
  {\it E-mail address} {\rm michael.ruzhansky@ugent.be}
 }

\author[Niyaz Tokmagambetov]{Niyaz Tokmagambetov}
\address{
  Niyaz Tokmagambetov:
  \endgraf
Department of Mathematics: Analysis,
Logic and Discrete Mathematics
  \endgraf
Ghent University, Belgium
  \endgraf
 and
 \endgraf
    Al--Farabi Kazakh National University
  \endgraf
  Almaty, Kazakhstan
  \endgraf
  {\it E-mail address} {\rm tokmagambetov@math.kz}
 }

\thanks{The authors were supported by FWO Odysseus 1 grant G.0H94.18N: Analysis and Partial Differential Equations. MR was also supported in parts by the EPSRC Grant EP/R003025/1, by the Leverhulme Research Grant RPG-2017-151.}

\date{\today}



\begin{abstract}
While it is known that one can consider the Cauchy problem for evolution equations with Caputo derivatives, the situation for the initial value problems for the Riemann-Liouville deri\-vatives is less understood.
In this paper we propose new type initial, inner and inner-boundary value problems for fractional differential equations with the Riemann-Liouville deri\-vatives. The results on the existence and uniqueness are proved, and conditions on the solvability are found. The well-posedness of the new type initial, inner and inner-boundary conditions are also discussed. Moreover, we give explicit formulas for the solutions. As an application fractional partial differential equations for general positive operators are studied.
\end{abstract}

\maketitle
\tableofcontents
\section{Introduction}
It is well-known that initial-value problems (Cauchy problem) for differential equations (DEs) play crucial role in studying different kinds of problems for differential equations and their applications. Therefore, fractional generalisations of integer order differential equations started with the consideration of modified Cauchy problems.

While it is known that one can consider the usual Cauchy initial conditions for evolution equations with Caputo derivatives, the situation for the initial value problems for the Riemann-Liouville deri\-vatives is less understood. The first aim of this paper is to discuss possible types of initial conditions for the basic equation
\begin{equation}\label{fde0}
\left(D_{0+}^\alpha u\right)(t)-m u(t)=f(t),\,\,\,0<t<T,
\end{equation}
for $1<\alpha\leq 2$, where $f(t)$ is a given function, $m,\,T$ are given real numbers, and
\begin{equation}
\label{rld0}
\left(D_{0+}^\alpha u\right)(t)=\dfrac{1}{\Gamma(n-\alpha)}\left(\dfrac{d}{dt}\right)^n\int\limits_0^t \dfrac{u(s)ds}{(t-s)^{\alpha-n+1}}
\end{equation}
is the Riemann-Liouville fractional derivative of order $\alpha$ of the function $u(t)$. Once this is understood, we apply the results for the corresponding partial differential (or more general operator) equations of the same type.

The idea of reducing the Cauchy problem for fractional differential equations to the Volterra integral equation was carried out by Pitcher and Sewel \cite{PS}. They considered the following Cauchy problem
\begin{equation*}
    \left\{
    \begin{array}{l}
        \left(D_{a+}^\alpha y\right)(x)=f\left[x,y(x)\right],\,\,0<\alpha<1,\\
        \, y(a)=0,\\
    \end{array}
    \right.
\end{equation*}
and proved the existence of the continuous solution $y(x)$ for the corresponding nonlinear integral equation
\begin{equation*}
    y(x)=\frac{1}{\Gamma(\alpha)}\int\limits_a^x\frac{f\left[t,y(t)\right]}{(x-t)^{1-\alpha}}dt,
\end{equation*}
under some conditions on $f$.

Here $\left(D_{a+}^\alpha y\right)(x)$ is the Riemann-Liouville fractional derivative of order $\alpha$ (see (\ref{rld})). One can find more detailed references in this regard in the survey paper by Kilbas and Trujillo \cite{KT}.

Here we would like to note only some works, where the Cauchy problems for differential equations with different kinds of fractional derivative were studied. In some cases, initial conditions have to be modified according to the used fractional derivatives. For instance, in \cite{Tom12} Tomovski investigated the following Cauchy problem:
\begin{equation*}
     \left\{
    \begin{array}{l}
        \left(D_{a+}^{\mu, \nu} y\right)(x)=f\left[x,y(x)\right],\,\,n-1<\mu\leq n,\,n\in \mathbb{N},\,0\leq \nu\leq 1,\\
        \lim\limits_{x\rightarrow a+}\frac{d^k}{dx^k}I_{a+}^{(n-\mu)(1-\nu)}y(x)=c_k,\,\,c_k\in\mathbb{R}, \,k=0,1,...,n-1,\\
    \end{array}
    \right.
\end{equation*}
where
\begin{equation*}
    \left(D_{a+}^{\mu, \nu} y\right)(x)=I_{a+}^{\nu(1-\mu)}\frac{d^n}{dx^n}I_{a+}^{(1-\nu)(n-\mu)}y(x)
\end{equation*}
is the right-hand side composite fractional derivative (Hilfer) of order $0<\mu<1$ and type $0\leq \nu\leq 1$ \cite{H09}, $I_{a+}^\alpha y(x)$ is the Riemann-Liouville fractional integral of order $\alpha$ (see (\ref{rli})).

In the works of the first author with his co-authors, due to consideration of the Caputo type fractional derivatives, they have used usual initial conditions, for instance, see \cite{F18} in the case of Hyper-Bessel fractional derivative and \cite{N18} in the case of fractional derivative with the Mittag-Leffler function in the kernel.

Changes in initial conditions might be motivated to avoid certain conditions on the orders of considered fractional differential equations. As Pskhu mentioned in his work \cite{pskhu11}, the Cauchy problem
\begin{equation*}
     \left\{
    \begin{array}{l}
        \left(D_{0+}^{\alpha} y\right)(x)-\lambda\left(D_{0+}^{\beta} y\right)(x) =0,\,1<\alpha<2,\,\beta<\alpha,\\
        \lim\limits_{x\rightarrow 0+}\left(D_{0+}^{\alpha-1} y\right)(x)=A,\,\,\lim\limits_{x\rightarrow 0+}\left(D_{0+}^{\alpha-2} y\right)(x)=B,\\
    \end{array}
    \right.
\end{equation*}
is ill-posed without the additional condition $\beta>\alpha-1$.

In general, the Cauchy problem for a multi-term fractional differential equation
\begin{equation*}
     \left\{
    \begin{array}{l}
        \left(D_{0+}^{\alpha} y\right)(x)-\sum\limits_{i=1}^m \lambda_i\left(D_{0+}^{\alpha_i} y\right)(x) =f(x),\\
        \lim\limits_{x\rightarrow 0+}\left(D_{0+}^{\alpha-k} y\right)(x)=A_k,\,\,1\leq k\leq n,\\
    \end{array}
    \right.
\end{equation*}
where $\alpha>0,\,\alpha\in (n-1,n],\,n\in\mathbb{N},\,\alpha>\alpha_i,\,\lambda_i\in\mathbb{R},\,i=1, \dots, n,\,x\in(0,l)$, would be well-posed if one imposes the condition $\alpha-n+1>\max\limits_{1\leq i \leq n}\{\alpha_i\}$.

In order to avoid this restriction, Pskhu suggested to use the following initial conditions
\begin{equation*}
    \lim\limits_{x\rightarrow 0+} \textbf{l}^ky(x)=a_k,\,\,\,1\leq k\leq n,
\end{equation*}
where
\begin{equation*}
    \textbf{l}^ky(x)=\left(D_{0+}^{\alpha-k} y\right)(x)-\sum\limits_{i=0}^{M(k)} \lambda_i\left(D_{0+}^{\alpha_i-k} y\right)(x),
\end{equation*}
\begin{equation*}
    M(k)=
    \left\{
    \begin{array}{l}
    0,\,\,\,\,if\,\,\,\,\alpha_i<\alpha-n+k,\,\,\textrm{ for all }\,i\\
    \max\limits_{\alpha_i-k\geq \alpha-n}\{i\},\,\,\,\textrm{ otherwise}.
    \end{array}
    \right.
\end{equation*}

We can also note some papers related with investigations of boundary value problems for ordinary differential equations with different fractional derivatives. In particular, in \cite{nakh} Sturm-Liouville problem for the second order differential equations with fractional derivatives in lower terms has been studied. Generalized Dirichlet and Neumann problems for multi-term fractional derivatives were subject of investigation in \cite{gadz1}, \cite{gadz2}. In \cite{mazh}, the author considered the following equation
\[
\partial_{0t}^\alpha u(t)-\lambda u(t)-\mu H(t-\tau)u(t-\tau)=f(t),\,\,0<t<1,
\]
where $\partial_{0t}^\alpha u(t)=\left(D_{0+}^{\alpha-n} u^{(n)}\right)(t)$, $H(t)$ is the Heaviside function, $n-1<\alpha\le n,$ $\lambda$, $\mu$ are any constants, $\tau$ is a fixed positive number.

A unique solvability of the Cauchy problem was proved. Moreover, in the case of $1<\alpha\le 2$, a boundary value problem with the following conditions
$$
au(0)+bu'(0)=0,\,\,\,\,cu(1)+du'(1)=0,
$$
where $a,b,c,d$ are given constants such that $a^2+b^2\neq 0$, $c^2+d^2\neq 0$, was investigated for the subject of the unique solvability under some conditions on parameters. A similar problem for a multi-term fractional differential equation with the Caputo derivative was studied in \cite{gadz3}. For more works related to this topic, we refer to the authors' papers \cite{KMR18, RTT20a, RTT20b}.

\section{Cauchy type problems}

In this paper our main target is the following integro-differential equation
\begin{equation}\label{fde}
\left(D_{0+}^\alpha u\right)(t)-m u(t)=f(t),\,\,\,0<t<T,
\end{equation}
for $1<\alpha\leq 2$, where $f(t)$ is a given function and $m,\,T$ are given real numbers. For $n:=[\alpha]+1$,
\begin{equation}
\label{rld}
\left(D_{0+}^\alpha u\right)(t)=\left(\dfrac{d}{dt}\right)^n\left(I_{0+}^{n-\alpha} u\right)(t)=\dfrac{1}{\Gamma(n-\alpha)}\left(\dfrac{d}{dt}\right)^n\int\limits_0^t \dfrac{u(s)ds}{(t-s)^{\alpha-n+1}}
\end{equation}
is the Riemann-Liouville fractional derivative of order $\alpha$ of the function $u(t)$. Here
\begin{equation}\label{rli}
\left(I_{0+}^{\alpha} u\right)(t)=\frac{1}{\Gamma(\alpha)}\int\limits_0^t \frac{u(s)ds}{(t-s)^{1-\alpha}}
\end{equation}
is the Riemann-Liouville fractional integral of order $0<\alpha<1$ of the function $u(t)$.

The general solution of the equation (\ref{fde}) is given by \cite{Kilbas} in the form
\begin{equation}\label{gs}
\begin{split}
u(t)=C_1t^{\alpha-1}E_{\alpha,\alpha}(mt^\alpha)&+C_2t^{\alpha-2}E_{\alpha,\alpha-1}(mt^\alpha)\\
&+\int\limits_0^t(t-s)^{\alpha-1}E_{\alpha,\alpha}(m(t-s)^\alpha)f(s)ds.
\end{split}
\end{equation}
Here $C_1$ and $C_2$ are unknown constants and
\begin{equation}\label{mlf}
E_{\alpha,\beta}(z)=\sum\limits_{k=0}^\infty \frac{z^k}{\Gamma(\alpha k+\beta)}
\end{equation}
is the two-parameter Mittag-Leffler function.

The most interesting properties of the two-parameter Mittag-Leffler function are associated with its asymptotic expansions as $z\rightarrow\infty,$ as shown in \cite[p.43]{Kilbas}:

\begin{lemma} \textsc{(\cite[p.43]{Kilbas})}
\label{C1: PrML}
Let $0<\alpha<2$ and $\pi\alpha/2<\mu<\min \{\pi,\pi \alpha\}.$ Then for $\mu\le \arg(z)\le \pi$, we have the following estimates
\begin{equation}
\label{Est-MLF}
\left|E_{\alpha, \beta}(z)\right|\le\frac{C}{1+|z|},
\end{equation}
where $C$ is a constant not depending on $z.$
\end{lemma}

\begin{defi} (\cite{Dim82, LG99})
\label{FS}
Let us consider the set of complex-valued functions $f$ defined on $[a, b]$. Fix $\gamma\geq-1$.
\begin{itemize}
\item We say that $f\in C_{\gamma}^{0}[a, b]:=C_{\gamma}[a, b]$,  if there is a real number $p>\gamma,$ such that
$$
f(x)=(x-a)^{p}f_{1}(x)
$$
with $f_{1}\in C[a, b]$.
\item We say that $f\in C_{\gamma}^{m}[a, b]$, $m\in\mathbb N_{0}$, if and only if
$$
f^{(m)}\in C_{\gamma}[a, b].
$$
\end{itemize}
From \cite{LG99} it follows that $C_{\gamma}[a, b]$ is a vector space and the set of spaces $C_{\gamma}[a, b]$ is ordered by inclusion according to
$$
C_{\gamma}[a, b]\subset C_{\beta}[a, b] \Leftrightarrow \gamma\geq\beta\geq-1.
$$
For further properties of $C_{\gamma}[a, b]$ we refer to \cite{LG99}.
\end{defi}

\subsection{Generalized initial (Cauchy) problem}

Our main aim is to find unknown constants of \eqref{gs} by using initial, inner and inner boundary conditions. First we start from the initial conditions.

Here, we are interested in finding solutions of the integro-differential equation \eqref{fde} satisfying the following initial conditions
\begin{equation}\label{initc}
\left.\frac{t^{2-\alpha-\beta}}{\Gamma(\alpha+\beta-1)}\left(I_{0+}^\beta u\right)(t)\right|_{t=0}=\hat{c_1},\,\left.\left(D_{0+}^\gamma u\right)(t)\right|_{t=0}=\hat{c_2},
\end{equation}
 where $\hat{c_1},\,\hat{c_2},\,\beta,\,\gamma$ are given real numbers such that $0\leq\beta<1,\,0<\gamma\leq 1$.

A motivation to consider generalized initial conditions of the type \eqref{initc} comes from the physical/technical point. Measuring data of the processes at the very short time is somehow impossible, and usually it takes a while. So, the models averaging this data and proposing its limit (trace-like) cases are crucial to better understand the whole process.

Using the integral operator \eqref{rli}, from the general form \eqref{gs} of the solution of the equation \eqref{fde} we obtain
\begin{equation*}
\begin{split}
\left(I_{0+}^{\beta} u\right)(t)&=\dfrac{C_1}{\Gamma(\beta)}\int\limits_0^t(t-s)^{\beta-1}s^{\alpha-1}E_{\alpha,\alpha}(ms^\alpha)ds
+\dfrac{C_2}{\Gamma(\beta)}\int\limits_0^t(t-s)^{\beta-1}s^{\alpha-2}E_{\alpha,\alpha-1}(ms^\alpha)ds \\
&+\dfrac{1}{\Gamma(\beta)}\int\limits_0^t(t-s)^{\beta-1}\left(\int\limits_0^s(s-z)^{\alpha-1}E_{\alpha,\alpha}(m(s-z)^\alpha)f(z)dz\right)ds.
\end{split}
\end{equation*}
Using \cite[p.25]{Podlubny}, for $\nu>0$ and $\beta>0$ we have
\begin{equation}\label{mlip}
\frac{1}{\Gamma(\nu)}\int\limits_0^{z}(z-s)^{\nu-1}s^{\beta-1}E_{\alpha,\beta}(\lambda s^\alpha)ds=z^{\beta+\nu-1}E_{\alpha,\beta+\nu}(\lambda z^\alpha),
\end{equation}
and also changing the order of integration in the latter integral, we find
\begin{equation*}
\begin{split}
\left(I_{0+}^{\beta} u\right)(t)&=C_1t^{\alpha+\beta-1}E_{\alpha, \alpha+\beta}(mt^\alpha)+C_2t^{\alpha+\beta-2}E_{\alpha, \alpha+\beta-1}(mt^\alpha)\\
&+\dfrac{1}{\Gamma(\beta)}\int\limits_0^tf(z)dz\int\limits_z^t(t-s)^{\beta-1}(s-z)^{\alpha-1}E_{\alpha,\alpha}(m(s-z)^\alpha)ds.
\end{split}
\end{equation*}
Since
\[
\frac{1}{\Gamma(\beta)}\int\limits_z^t (t-s)^{\beta-1}(s-z)^{\alpha-1}E_{\alpha,\alpha}(m(s-z)^\alpha)ds=(t-z)^{\alpha+\beta-1}E_{\alpha,\alpha+\beta}(m(t-z)^\alpha),
\]
we get
\begin{equation}\label{fic1}
\begin{split}
\left(I_{0+}^{\beta} u\right)(t)&=C_1t^{\alpha+\beta-1}E_{\alpha, \alpha+\beta}(mt^\alpha)+C_2t^{\alpha+\beta-2}E_{\alpha, \alpha+\beta-1}(mt^\alpha)\\
&+\int\limits_0^tf(t-s)s^{\alpha+\beta-1}E_{\alpha,\alpha+\beta}(ms^\alpha)ds.
\end{split}
\end{equation}

By letting the integro-differential operator \eqref{rld} act on the general solution \eqref{gs} of the equation \eqref{fde}, we get
\begin{equation*}
\begin{split}
\left.\left(D_{0+}^\gamma u\right)(t)\right|_{t=0}
&=C_1D_{0+}^\gamma\left[t^{\alpha-1}E_{\alpha,\alpha}(mt^\alpha)\right]+C_2D_{0+}^\gamma\left[t^{\alpha-2}E_{\alpha,\alpha-1}(mt^\alpha)\right]\\
&+D_{0+}^\gamma\left[\int\limits_0^t (t-s)^{\alpha-1}E_{\alpha,\alpha}(m(t-s)^\alpha)f(s)ds\right].
\end{split}
\end{equation*}
Using \cite[p.21]{Podlubny}, \begin{equation}\label{fdml}
D_{0+}^\gamma \left[s^{\beta-1}E_{\alpha,\beta}(\lambda s^\alpha)\right]=s^{\beta-\gamma-1}E_{\alpha,\beta-\gamma}(\lambda s^\alpha), \,\,\, \hbox{for} \,\,\, \alpha>0, \beta>\gamma,
\end{equation}
we obtain
\begin{equation*}
D_{0+}^\gamma\left[t^{\alpha-1}E_{\alpha,\alpha}(mt^\alpha)\right]=t^{\alpha-\gamma-1}E_{\alpha,\alpha-\gamma}(mt^\alpha),
\end{equation*}
\begin{equation}\label{D_gamma}
D_{0+}^\gamma\left[t^{\alpha-2}E_{\alpha,\alpha-1}(mt^\alpha)\right]=t^{\alpha-\gamma-2}E_{\alpha,\alpha-1-\gamma}(mt^\alpha),
\end{equation}
and
\begin{equation*}
\begin{split}
D_{0+}^\gamma&\left[\int\limits_0^t (t-s)^{\alpha-1}E_{\alpha,\alpha}(ms^\alpha)f(s)ds\right]
\\
&=\dfrac{1}{\Gamma(1-\gamma)}\dfrac{d}{dt}\int\limits_0^t(t-s)^{-\gamma}\left(\int\limits_0^s(s-z)^{\alpha-1}E_{\alpha,\alpha}(m(s-z)^\alpha)f(z)dz\right)ds
\\
&=\dfrac{1}{\Gamma(1-\gamma)}\dfrac{d}{dt}\int\limits_0^t f(z)dz\int\limits_z^t(t-s)^{-\gamma}(s-z)^{\alpha-1}E_{\alpha,\alpha}(m(s-z)^\alpha)ds
\\
&=\dfrac{1}{\Gamma(1-\gamma)}\dfrac{d}{dt}\int\limits_0^t f(z)dz\int\limits_0^{t-z}(t-z-\xi)^{1-\gamma-1}\xi^{\alpha-1}E_{\alpha,\alpha}(m\xi^\alpha)d\xi
\\
&=\dfrac{d}{dt}\int\limits_0^t f(z)(t-z)^{\alpha-\gamma}E_{\alpha,\alpha-\gamma+1}(m(t-z)^\alpha)dz
\\
&=f(t)(t-t)^{\alpha-\gamma}E_{\alpha,\alpha-\gamma+1}(m(t-t)^\alpha)+\int\limits_0^tf(z)\dfrac{d}{dt}\left[(t-z)^{\alpha-\gamma}E_{\alpha,\alpha-\gamma+1}(m(t-z)^\alpha)\right]dz
\\
&=\int\limits_0^tf(t-z)z^{\alpha-\gamma-1}E_{\alpha,\alpha-\gamma}(mz^\alpha)dz.
\end{split}
\end{equation*}

In what follows, we will need a special critical case of \eqref{D_gamma}, namely, when $\gamma=\alpha-1$. For this, we calculate
$$
D_{0+}^{\alpha-1}\left[t^{\alpha-2}E_{\alpha,\alpha-1}(mt^\alpha)\right] = \dfrac{1}{\Gamma(2-\alpha)}\dfrac{d}{dt}\int\limits_0^t(t-s)^{1-\alpha} s^{\alpha-2}E_{\alpha,\alpha-1}(m s^\alpha)ds
$$
by using \eqref{mlip} when $\nu=2-\alpha, \beta=\alpha-1, \lambda=m$. Thus, we get
\begin{equation}\label{D_alpha-1}
\begin{split}
D_{0+}^{\alpha-1}\left[t^{\alpha-2}E_{\alpha,\alpha-1}(mt^\alpha)\right] &=\dfrac{d}{dt} \left[E_{\alpha,1}(mt^\alpha)\right]
\\
&=\dfrac{d}{dt} \left[\sum\limits_{k=0}^{\infty}\frac{(mt^\alpha)^{k}}{\Gamma(\alpha k+1)}\right]
\\
&=mt^{\alpha-1} E_{\alpha, \alpha}(mt^\alpha).
\end{split}
\end{equation}
Now, for formality, by using the formula $E_{\alpha, \beta}(z)-zE_{\alpha, \beta}(z)=\frac{1}{\Gamma(\beta)}$ when $\beta=0$ and $z=mt^\alpha$, we obtain
\begin{equation}\label{D0}
\begin{split}
D_{0+}^{\alpha-1}\left[t^{\alpha-2}E_{\alpha,\alpha-1}(mt^\alpha)\right] = \frac{1}{t} E_{\alpha, 0}(mt^\alpha).
\end{split}
\end{equation}
Here we take into account that $\frac{1}{\Gamma(\beta)}\to 0$ as $\beta\to 0$. Note that the two-parameter Mittag-Leffler functions $E_{\alpha, \beta}(z)$ for complex-valued $\beta$ are well-defined and studied, for example, in the book \cite{GKMR14}.

Hence
\begin{equation}\label{fic2}
\begin{split}
\left(D_{0+}^\gamma u\right)(t)=&C_1t^{\alpha-\gamma-1}E_{\alpha,\alpha-\gamma}(mt^\alpha)+C_2t^{\alpha-\gamma-2}E_{\alpha,\alpha-1-\gamma}(mt^\alpha)\\
&+\int\limits_0^tf(t-s)s^{\alpha-\gamma-1}E_{\alpha,\alpha-\gamma}(ms^\alpha)ds,
\end{split}
\end{equation}
for all $\gamma\leq\alpha-1$. Considering (\ref{fic1}) and the first condition of \eqref{initc}, if $2-\alpha-\beta\ge 0$ we get
\begin{equation*}
\begin{split}
&\left.\frac{1}{\Gamma(\alpha+\beta-1)}t^{2-\alpha-\beta}\left(I_{0+}^\beta u\right)(t)\right|_{t=0}\\
&=\left[C_1tE_{\alpha,\alpha+\beta}(mt^\alpha)+C_2E_{\alpha,\alpha+\beta-1}(mt^\alpha)
\left.+t^{2-\alpha-\beta}\int\limits_0^tf(t-s)s^{\alpha+\beta-1}E_{\alpha,\alpha+\beta}(ms^\alpha)ds\right]\right|_{t=0}\\
&=
\hat{c_1},
\end{split}
\end{equation*}
which yields $C_2=\hat{c_1}$.

Now we substitute (\ref{fic2}) into the second condition of (\ref{initc}), that is, one obtains
\begin{equation}\label{ev1}
\begin{split}
&\left.\left(D_{0+}^\gamma u\right)(t)\right|_{t=0}\\
&=\left[C_1t^{\alpha-\gamma-1}E_{\alpha,\alpha-\gamma}(mt^\alpha)+C_2t^{\alpha-\gamma-2}E_{\alpha,\alpha-1-\gamma}(mt^\alpha)
\left.+\int\limits_0^tf(t-s)s^{\alpha-\gamma-1}E_{\alpha,\alpha-\gamma}(ms^\alpha)ds\right]\right|_{t=0}\\
&=\hat{c_2}.
\end{split}
\end{equation}
In order to find $C_1$ we have to suppose that $\gamma=\alpha-1$,
otherwise we would not be able to find $C_1$.
Indeed, the first term in the right-hand side of (\ref{ev1}) becomes zero due to $\alpha-1-\gamma>0$ as $t=0$, and $\dfrac{t^{\alpha-\gamma-2}}{\Gamma(\alpha-1-\gamma)}$ has a singularity at point $t=0$. Hence, if  $\gamma= \alpha-1$
from (\ref{ev1}) it follows that $C_1=\hat{c_2}$.

Finally, the solution of (\ref{fde}) satisfying the initial conditions (\ref{initc}) when $\gamma= \alpha-1$ has the form
\begin{equation}\label{sic}
\begin{split}
u(t)=\hat{c_2}t^{\alpha-1}E_{\alpha,\alpha}(mt^\alpha)&+\hat{c_1}t^{\alpha-2}E_{\alpha,\alpha-1}(mt^\alpha)
\\
&+\int\limits_0^tf(t-s)s^{\alpha-1}E_{\alpha,\alpha}(ms^\alpha)ds.
\end{split}
\end{equation}

The following theorem holds true.
\begin{thm}
\label{Th1}
Assume that $1<\alpha\leq 2$ and
$0\leq\beta\le 2-\alpha$.
\begin{itemize}
    \item[\textbf{(A)}] Let $0<\gamma\leq\alpha-1$. Suppose $f\in C_{-1}^{1}[0,T]$ if $1<\alpha<2$ and $f\in C_{-1}[0,T]$ if $\alpha=2$.
Then the fractional order differential equation \eqref{fde} satisfying the homogeneous initial conditions \eqref{initc} with $\hat{c_1}=\hat{c_2}=0$ has a unique solution $u_f\in C_{-1}^{2}[0,T]$ of the form \eqref{sic}, namely, in our case we have
\begin{equation}
\label{sic-2}
u_f(t)=\int\limits_0^tf(t-s)s^{\alpha-1}E_{\alpha,\alpha}(ms^\alpha)ds.
\end{equation}

\item[\textbf{(B)}] Let $\gamma=\alpha-1$ and $f\equiv0$. Then the solution of the equation \eqref{fde} satisfying the non-homogeneous initial conditions \eqref{initc} has a unique solution $u_0$ of the following form
\begin{equation}\label{sic-3}
\begin{split}
u_0(t)=\hat{c_2}t^{\alpha-1}E_{\alpha,\alpha}(mt^\alpha)&+\hat{c_1}t^{\alpha-2}E_{\alpha,\alpha-1}(mt^\alpha).
\end{split}
\end{equation}
Moreover, we have
\begin{equation}\label{initc-2}
\begin{split}
\left.\left(\frac{t^{2-\alpha-\beta}}{\Gamma(\alpha+\beta-1)}I_{0+}^\beta [t^{\alpha-1}E_{\alpha,\alpha}(mt^\alpha)]\right)\right|_{t=0}&=0,\\
\left.\left(D_{0+}^\gamma [t^{\alpha-1}E_{\alpha,\alpha}(mt^\alpha)]\right)\right|_{t=0}&=1,
\end{split}
\end{equation}
and
\begin{equation}\label{initc-3}
\begin{split}
\left.\left(\frac{t^{2-\alpha-\beta}}{\Gamma(\alpha+\beta-1)}I_{0+}^\beta [t^{\alpha-2}E_{\alpha,\alpha-1}(mt^\alpha)]\right)\right|_{t=0}&=1,\\
\left.\left(D_{0+}^\gamma [t^{\alpha-2}E_{\alpha,\alpha-1}(mt^\alpha)]\right)\right|_{t=0}&=0.
\end{split}
\end{equation}
\end{itemize}
\end{thm}
\begin{proof} \textbf{(A)} Here, we briefly give some ideas of the proof. First of all, we note that the uniqueness of the solution to the equation \eqref{fde} with the initial conditions \eqref{initc} follows from the linearity of the equation in combination with the general arguments. The representation \eqref{sic-2} of the formal solution $u$ yields the existence. From the paper of Luchko and Gorenflo \cite[Theorem 4.1]{LG99} it follows that the formal solution \eqref{sic-2} $u$ is from $C_{-1}^{2}[0,T]$. The last but not less important thing is that the functionals generating the initial conditions \eqref{initc} are well-defined due to \cite[Theorem 2.2]{LG99}.

\textbf{(B)} The existence and uniqueness results are clear. The properties \eqref{initc-2} and \eqref{initc-3} follow from the previous calculations.
\end{proof}

\begin{rem}
Let $0\leq\gamma<\alpha-1$. Then for
$$
f(t)=t^{\alpha-2}E_{\alpha,\alpha-1}(mt^\alpha),
$$
we have
$$
\left.\left(D_{0+}^\gamma [f(t)]\right)\right|_{t=0}=\infty.
$$
\end{rem}

\begin{rem}
If $\beta=2-\alpha,\,\gamma=\alpha-1$, then (\ref{initc}) yields the classical initial conditions for the equation (\ref{fde}).
\end{rem}

\subsection{Partial differential equations with the Cauchy type data}

Here we consider the Cauchy type problem for the time--fractional wave equation
\begin{equation}\label{C1: fde}
\left(D_{0+}^\alpha u\right)(t) + \mathcal A u(t)=f(t),\,\,\, 0<t<T,
\end{equation}
with initial type conditions
\begin{equation}\label{C1: initc}
\left.\left(\frac{t^{2-\alpha-\beta}}{\Gamma(\alpha+\beta-1)}I_{0+}^\beta u\right)(t)\right|_{t=0}=u_{1},\,\left.\left(D_{0+}^\gamma u\right)(t)\right|_{t=0}=u_{2},
\end{equation}
where $1<\alpha\leq 2$, $\mathcal A$ is a self--adjoint operator on the separable Hilbert space $\mathbb H$, $u_1$ and $u_2$ are the initial functions.

\begin{assumption}
\label{Assumption_Sp-Oper}
We assume that the operator $\mathcal A$ has a positive discrete spectrum $\{m_{\xi}\}_{\xi\in\mathbb N}$ such that
$$
\inf\limits_{\xi\in\mathbb N} m_\xi>0
$$
with the corresponding system of eigenfunctions $\{e_{\xi}\}_{\xi\in\mathbb N}$ forming an orthonormal basis in $\mathbb H$. We denote by $(\cdot, \cdot)_{\mathbb H}$ inner product of the Hilbert space $\mathbb H.$
\end{assumption}

Let us introduce $\chi(t)=t^{2(2-\alpha)}$ for $t\in[0, 1)$ and $\chi(t)=1$ for $t\in[1, T]$. Then, we define
$$
\|u\|_{L_{\chi}^{2}(0, T; \, \mathbb H)}:=\left(\int\limits_{0}^{T}\|u(s)\|_{\mathbb H}^{2} \chi(s) ds \right)^{1/2},
$$
where $\|\cdot\|_{\mathbb H}$ is the norm of the Hilbert space $\mathbb H$.

\begin{thm}
\label{Th_appl_1}
Let $1<\alpha\leq2$. Let $\gamma = \alpha-1$ and $0\leq\beta\le 2-\alpha$.
\begin{itemize}
    \item[\textbf{(A)}] Assume that $f\in L^{2}(0, T; \, \mathbb H)$. Suppose that $u_1, u_2\in\mathbb H$. Then the Cauchy type problem for the time-fractional wave equation \eqref{C1: fde}--\eqref{C1: initc} has a unique solution $u\in L_{\chi}^{2}(0, T; \, \mathbb H)$ such that $\mathcal A u, D_{0+}^\alpha u \in L_{\chi}^{2}(0, T; \, \mathbb H)$ in the form
\begin{equation*}
\begin{split}
u(t)=t^{\alpha-2}&\sum\limits_{\xi\in\mathbb N}u_{1\xi}E_{\alpha,\alpha-1}(-m_{\xi}t^\alpha)e_{\xi} + t^{\alpha-1}\sum\limits_{\xi\in\mathbb N}u_{2\xi}E_{\alpha,\alpha}(-m_{\xi}t^\alpha)e_{\xi}
\\
&+\sum\limits_{\xi\in\mathbb N}\left[\int\limits_0^t s^{\alpha-1}f_{\xi}(t-s)E_{\alpha,\alpha}(-m_{\xi}s^\alpha) ds\right]e_{\xi},
\end{split}
\end{equation*}
where
$$
f_{\xi}=(f, e_\xi)_{\mathbb H}, \,\,\,u_{k\xi}=(u_k, e_\xi)_{\mathbb H}
$$
for $k=1, 2$, for all $\xi\in\mathbb N$.

    \item[\textbf{(B)}]
    Assume that $f\in C_{-1}([0,T]; \mathbb{H})$ if $\alpha=2$, and $f\in C_{-1}^{1}([0,T]; \mathbb{H})$ if $1<\alpha<2$. Suppose that $u_1\equiv0$ and $u_2\equiv0$. Then the Cauchy type problem for the time-fractional wave equation \eqref{C1: fde}--\eqref{C1: initc} has a unique solution $u_f\in C_{-1}^{2}([0, T]; \, \mathbb H)$ such that $\mathcal A u_f \in C_{-1}([0, T]; \, \mathbb H)$.
\end{itemize}
\end{thm}
\begin{proof} \textbf{(A)} Let us seek a solution of the problem \eqref{C1: fde}--\eqref{C1: initc} in the following form
$$
u(t)=\sum\limits_{\xi\in\mathbb N} u_{\xi}(t) e_{\xi},
$$
with the source term and initial data
$$
f(t)=\sum\limits_{\xi\in\mathbb N} f_{\xi}(t) e_{\xi}, \,\,\, u_{1}=\sum\limits_{\xi\in\mathbb N} u_{1\xi} e_{\xi}, \,\,\, u_{2}=\sum\limits_{\xi\in\mathbb N} u_{2\xi} e_{\xi},
$$
respectively.

By putting them into the equation \eqref{C1: fde}, we obtain a set of the Cauchy type problems for each $\xi\in\mathbb N$:
\begin{equation}\label{C1: EQ1}
\left(D_{0+}^\alpha u_{\xi} \right)(t) + m_{\xi} u_{\xi}(t)=f_{\xi}(t),\,\,\, t>0,
\end{equation}
with initial type conditions
\begin{equation}\label{C1: CP1}
\left.\left(\frac{t^{2-\alpha-\beta}}{\Gamma(\alpha+\beta-1)} I_{0+}^\beta u_{\xi}\right)(t)\right|_{t=0}=u_{1\xi},\,\left.\left(D_{0+}^\gamma u_{\xi}\right)(t)\right|_{t=0}=u_{2\xi},
\end{equation}
where $1<\alpha\leq 2$, $m_{\xi}$ is an eigenvalue of the self--adjoint operator $\mathcal A$ corresponding to the eigenfunction $e_{\xi}$ for all $\xi\in\mathbb N$, and $\beta,\,\gamma$ are given real numbers such that $0\leq\beta<1,\,0<\gamma\leq 1$.

A general solution of the equation \eqref{C1: EQ1} satisfying the initial type conditions \eqref{C1: CP1} at $\gamma=\alpha-1$ has the form
\begin{equation}\label{C1: preSol1}
\begin{split}
 u_{\xi}(t)=u_{2\xi}t^{\alpha-1}&E_{\alpha,\alpha}(-m_{\xi}t^\alpha)+u_{1\xi}t^{\alpha-2}E_{\alpha,\alpha-1}(-m_{\xi}t^\alpha)\\
&+\int\limits_0^t f_{\xi}(t-s)s^{\alpha-1}E_{\alpha,\alpha}(-m_{\xi}s^\alpha)ds.
\end{split}
\end{equation}

Thus, we have
\begin{equation}\label{C1: Sol1}
\begin{split}
 u(t)&=t^{\alpha-2}\sum\limits_{\xi\in\mathbb N}u_{1\xi}E_{\alpha,\alpha-1}(-m_{\xi}t^\alpha)e_{\xi} \\
 &+ t^{\alpha-1}\sum\limits_{\xi\in\mathbb N}u_{2\xi}E_{\alpha,\alpha}(-m_{\xi}t^\alpha)e_{\xi}\\
&+\int\limits_0^t s^{\alpha-1} \left[\sum\limits_{\xi\in\mathbb N}f_{\xi}(t-s)E_{\alpha,\alpha}(-m_{\xi}s^\alpha)e_{\xi}\right] ds.
\end{split}
\end{equation}
Now, by taking $\mathbb H$-norm, for all $t\in[0, T]$ we obtain
\begin{equation*}
\begin{split}
\|u(t)\|_{\mathbb H}^{2} &\leq C \sum\limits_{\xi\in\mathbb N}|t^{\alpha-2}E_{\alpha,\alpha-1}(-m_{\xi}t^\alpha)|^{2}  |u_{1\xi}|^{2} \\
 &+ \sum\limits_{\xi\in\mathbb N}|t^{\alpha-1}E_{\alpha,\alpha}(-m_{\xi}t^\alpha)|^{2}  |u_{2\xi}|^{2} \\
&+\sum\limits_{\xi\in\mathbb N}\left|\int\limits_0^t |s^{\alpha-1} E_{\alpha,\alpha}(-m_{\xi}s^\alpha)| |f_{\xi}(t-s)| d s \right|^{2}
\end{split}
\end{equation*}
\begin{equation}\label{C1: Sol1-mod}
\begin{split}
&\leq C \sum\limits_{\xi\in\mathbb N}\left(\frac{t^{\alpha-2}}{1+m_{\xi}t^\alpha}\right)^{2} |u_{1\xi}|^{2}  + C \sum\limits_{\xi\in\mathbb N}\left(\frac{t^{\alpha-1}}{1+m_{\xi}t^\alpha}\right)^{2} |u_{2\xi}|^{2}\\
& + C \sum\limits_{\xi\in\mathbb N}\left(\int\limits_0^t \frac{s^{\alpha-1}}{1+m_{\xi}s^\alpha} |f_{\xi}(t-s)| d s\right)^{2}
\\
&\leq C \sum\limits_{\xi\in\mathbb N}\left(\frac{t^{\alpha-2}}{1+m_{\xi}t^\alpha}\right)^{2} |u_{1\xi}|^{2}  + C \sum\limits_{\xi\in\mathbb N}\left(\frac{t^{\alpha-1}}{1+m_{\xi}t^\alpha}\right)^{2} |u_{2\xi}|^{2}\\
& + C \sum\limits_{\xi\in\mathbb N}\int\limits_0^t \left(\frac{s^{\alpha-1}}{1+m_{\xi}s^\alpha}\right)^{2}  d s
\int\limits_0^t |f_{\xi}(t-s)|^{2} d s.
\end{split}
\end{equation}
Recall that $\chi(t)=t^{2(2-\alpha)}$ for $t\in[0, 1)$ and $\chi(t)=1$ for $t\in[1, T]$. Then, by using
$$
\frac{t^{\alpha-k}}{1+m_{\xi}t^\alpha}=\frac{1}{t^{k-\alpha}+m_{\xi}t^k}\leq\frac{C}{1+t^k},
$$
for $t\geq1$, for $k=1,2$, for some constant $C>0$, we get
$$
\chi(t)\|u(t)\|_{\mathbb H}^{2} \leq C\chi(t)\left( \sum\limits_{\xi\in\mathbb N}\left(\frac{t^{\alpha-2}}{1+m_{\xi}t^\alpha}\right)^{2} |u_{1\xi}|^{2}  +  \sum\limits_{\xi\in\mathbb N}\left(\frac{t^{\alpha-1}}{1+m_{\xi}t^\alpha}\right)^{2} |u_{2\xi}|^{2}\right)
$$
$$
+ C\chi(t) \sum\limits_{\xi\in\mathbb N}\int\limits_0^t \left(\frac{s^{\alpha-1}}{1+m_{\xi}s^\alpha}\right)^{2}  d s
\int\limits_0^t |f_{\xi}(t-s)|^{2} d s
$$
\begin{equation}\label{C1: Sol1-mod-weight}
\begin{split}
\leq C\left(
\sum\limits_{\xi\in\mathbb N} |u_{1\xi}|^{2}  + 
\sum\limits_{\xi\in\mathbb N} |u_{2\xi}|^{2} +
\int\limits_0^t \sum\limits_{\xi\in\mathbb N} |f_{\xi}(s)|^{2} d s \right),
\end{split}
\end{equation}
for all  $t\in[0, T]$. Thus, we have
\begin{equation}\label{C1: Sol1-L2(H)}
\begin{split}
\|u\|_{L_{\chi}^{2}(0, T; \, \mathbb H)}^{2}&:=\int\limits_{0}^{T}\|u(s)\|_{\mathbb H}^{2} \chi(s) ds
\\
& \leq C\left(\|u_{1}\|_{\mathbb H}^{2} + \|u_{2}\|_{\mathbb H}^{2} + \|f\|_{L^{2}(0, T; \, \mathbb H)}^{2} \right).
\end{split}
\end{equation}

Now, for the second term of the equation \eqref{C1: fde} we have
\begin{equation}\label{C1: Sol1-A}
\begin{split}
 \mathcal A u(t)&=t^{\alpha-2}\sum\limits_{\xi\in\mathbb N}u_{1\xi}E_{\alpha,\alpha-1}(-m_{\xi}t^\alpha) m_{\xi} e_{\xi} \\
 &+ t^{\alpha-1}\sum\limits_{\xi\in\mathbb N}u_{2\xi}E_{\alpha,\alpha}(-m_{\xi}t^\alpha)m_{\xi} e_{\xi}\\
&+\int\limits_0^t s^{\alpha-1} \left[\sum\limits_{\xi\in\mathbb N}f_{\xi}(t-s)E_{\alpha,\alpha}(-m_{\xi}s^\alpha) m_{\xi} e_{\xi}\right] ds.
\end{split}
\end{equation}
By taking $\mathbb H$-norm, we get
\begin{equation}\label{C1: Sol1-mod-A}
\begin{split}
\|\mathcal A u(t)\|_{\mathbb H}^{2} &\leq C \sum\limits_{\xi\in\mathbb N}|t^{\alpha-2} m_{\xi} E_{\alpha,\alpha-1}(-m_{\xi}t^\alpha)|^{2}  |u_{1\xi}|^{2} \\
 &+ \sum\limits_{\xi\in\mathbb N}|t^{\alpha-1} m_{\xi} E_{\alpha,\alpha}(-m_{\xi}t^\alpha)|^{2}  |u_{2\xi}|^{2} \\
&+\sum\limits_{\xi\in\mathbb N}\left|\int\limits_0^t |s^{\alpha-1}  m_{\xi} E_{\alpha,\alpha}(-m_{\xi}s^\alpha)| |f_{\xi}(t-s)| d s \right|^{2}
\\
&\leq C \sum\limits_{\xi\in\mathbb N}\left(\frac{m_{\xi} t^{\alpha-2}}{1+m_{\xi}t^\alpha}\right)^{2} |u_{1\xi}|^{2}  + C \sum\limits_{\xi\in\mathbb N}\left(\frac{m_{\xi} t^{\alpha-1}}{1+m_{\xi}t^\alpha}\right)^{2} |u_{2\xi}|^{2}\\
& + C \sum\limits_{\xi\in\mathbb N}\int\limits_0^t \left(\frac{m_{\xi} s^{\alpha-1}}{1+m_{\xi}s^\alpha}\right)^{2}d s \int\limits_0^t |f_{\xi}(t-s)|^{2} d s.
\end{split}
\end{equation}

Then, we obtain
\begin{equation}\label{C1: Sol1-L2(H)-A}
\begin{split}
\|\mathcal A u\|_{L_{\chi}^{2}(0, T; \, \mathbb H)}^{2}& \leq C\left(\|u_{1}\|_{\mathbb H}^{2} + \|u_{2}\|_{\mathbb H}^{2} + \|f\|_{L^{2}(0, T; \, \mathbb H)}^{2} \right).
\end{split}
\end{equation}

Finally, in a similar way, one can show that
\begin{equation}\label{C1: Sol1-L2(H)-U_t}
\begin{split}
\|D_{0+}^\alpha u\|_{L_{\chi}^{2}(0, T; \, \mathbb H)}^{2}& \leq C\left(\|u_{1}\|_{\mathbb H}^{2} + \|u_{2}\|_{\mathbb H}^{2} + \|f\|_{L^{2}(0, T; \, \mathbb H)}^{2} \right).
\end{split}
\end{equation}

\textbf{(B)} The proof of this part comes from combination of Theorem \ref{Th1} and from the arguments of Part \textbf{(A)}.
\end{proof}

\section{Inner value problem}

In this Section we study a non-local type Cauchy problem. We consider the equation \eqref{fde} with the following non-local conditions
\begin{equation}
\label{innc}
\begin{split}
\left.\left(I_{0+}^\beta u\right)(t)\right|_{t=a}=\hat{d_1},\,\,0\leq\beta\leq2-\alpha,
\\
\left.\left(D_{0+}^\gamma u\right)(t)\right|_{t=a}=\hat{d_2},\,\, 0<\gamma\leq\alpha-1,
\end{split}
\end{equation}
instead of the Cauchy type initial data \eqref{initc}. Here $0<a<T$, $\hat{d_1}$ and $\hat{d_2}$ are given real numbers.

Now we are in a position to construct a solution of the time-fractional equation \eqref{fde} satisfying the non-local conditions \eqref{innc}. Using \eqref{fic1} and \eqref{fic2}, we find
\begin{equation}
\label{finc1}
\begin{split}
\left.\left(I_{0+}^{\beta} u\right)(t)\right|_{t=a}&=C_1a^{\alpha+\beta-1}E_{\alpha, \alpha+\beta}(ma^\alpha)+C_2a^{\alpha+\beta-2}E_{\alpha, \alpha+\beta-1}(ma^\alpha)
\\
&+\int\limits_0^af(a-s)s^{\alpha+\beta-1}E_{\alpha,\alpha+\beta}(ms^\alpha)ds=\hat{d_1},
\end{split}
\end{equation}
\begin{equation}\label{finc2}
\begin{split}
\left.\left(D_{0+}^\gamma u\right)(t)\right|_{t=a}
&=C_1a^{\alpha-\gamma-1}E_{\alpha,\alpha-\gamma}(ma^\alpha)+C_2a^{\alpha-\gamma-2}E_{\alpha,\alpha-1-\gamma}(ma^\alpha)\\
&+\int\limits_0^af(a-s)s^{\alpha-\gamma-1}E_{\alpha,\alpha-\gamma}(ms^\alpha)ds=\hat{d_2}.
\end{split}
\end{equation}
The equalities (\ref{finc1})-(\ref{finc2}) form a system of algebraic equations with respect to unknown constants $C_1$ and $C_2$. This system will have a solution if
\begin{equation}\label{inncon}
E_{\alpha,\alpha+\beta}(ma^\alpha)E_{\alpha,\alpha-1-\gamma}(ma^\alpha)\neq E_{\alpha,\alpha+\beta-1}(ma^\alpha)E_{\alpha,\alpha-\gamma}(ma^\alpha).
\end{equation}
If this holds, a solution of the problem (\ref{fde}), (\ref{innc}) can be written as
\begin{equation}
\label{inns}
\begin{split}
u(t)&=\hat d_1\hat{E_1}(t)+\hat d_2\hat{E_2}(t)
\\
&+\int\limits_0^af(a-s)s^{\alpha-1}\left[s^\beta E_{\alpha,\alpha+\beta}(ms^\alpha)\hat{E_1}(m t)-s^{-\gamma}E_{\alpha,\alpha-\gamma}(ms^\alpha)\hat{E_2}(m t)\right]ds
\\
&+\int\limits_0^tf(t-s)s^{\alpha-1}E_{\alpha,\alpha}(ms^\alpha)ds,
\end{split}
\end{equation}
where
\begin{equation}
\label{eve}
\begin{split}
\hat{E_1}(m t)&=\dfrac{t^{\alpha-2}\left[tE_{\alpha,\alpha-1-\gamma}(ma^\alpha)E_{\alpha,\alpha}(mt^\alpha)-aE_{\alpha,\alpha-\gamma}(ma^\alpha)E_{\alpha,\alpha-1}(mt^\alpha)\right]}{a^{\alpha+\beta-1}\left[E_{\alpha,\alpha+\beta}(ma^\alpha)E_{\alpha,\alpha-1-\gamma}(ma^\alpha)- E_{\alpha,\alpha+\beta-1}(ma^\alpha)E_{\alpha,\alpha-\gamma}(ma^\alpha)\right]},
\\
\hat{E_2}(m t)&=\dfrac{t^{\alpha-2}\left[-tE_{\alpha,\alpha-1+\beta}(ma^\alpha)E_{\alpha,\alpha}(mt^\alpha)+aE_{\alpha,\alpha+\beta}(ma^\alpha)E_{\alpha,\alpha-1}(mt^\alpha)\right]}{a^{\alpha-\gamma-1}\left[E_{\alpha,\alpha+\beta}(ma^\alpha)E_{\alpha,\alpha-1-\gamma}(ma^\alpha)- E_{\alpha,\alpha+\beta-1}(ma^\alpha)E_{\alpha,\alpha-\gamma}(ma^\alpha)\right]}.
\end{split}
\end{equation}

These observations in combination with Theorem \ref{Th1} give the following result:
\begin{thm}
\label{Th2}
Assume that $1<\alpha\leq 2$ and
$0\leq\beta\le 2-\alpha$. Let $0<\gamma\leq\alpha-1$. Also, assume that the condition \eqref{inncon} holds.
\begin{itemize}
    \item[\textbf{(A)}] Suppose $f\in C_{-1}^{1}[0,T]$ if $1<\alpha<2$ and $f\in C_{-1}[0,T]$ if $\alpha=2$.
Then the fractional order differential equation \eqref{fde} satisfying the homogeneous non-local initial conditions \eqref{innc} with $\hat{d_1}=\hat{d_2}=0$ has a unique solution $u_f\in C_{-1}^{2}[0,T]$ of the form \eqref{inns}, namely, we have
\begin{equation}
\label{inns-2}
\begin{split}
u_f(t)&=-\int\limits_0^af(a-s)s^{\alpha-1}\left[s^\beta E_{\alpha,\alpha+\beta}(ms^\alpha)\hat{E_1}(t)+s^{-\gamma}E_{\alpha,\alpha-\gamma}(ms^\alpha)\hat{E_2}(t)\right]ds
\\
&+\int\limits_0^tf(t-s)s^{\alpha-1}E_{\alpha,\alpha}(ms^\alpha)ds.
\end{split}
\end{equation}

\item[\textbf{(B)}] Let $f\equiv0$. Then the solution of the homogeneous equation \eqref{fde} satisfying the non-homogeneous non-local initial conditions \eqref{innc} has a unique solution $u_0$ of the following form
\begin{equation}\label{inns-3}
\begin{split}
u_0(t)=\hat d_1\hat{E_1}(mt)+\hat d_2\hat{E_2}(mt).
\end{split}
\end{equation}
Moreover, we have
\begin{equation}
\label{innc-2}
\begin{split}
\left.\left(I_{0+}^\beta [\hat{E_1}(mt)]\right)\right|_{t=a}=1,\,\,
\left.\left(D_{0+}^\gamma [\hat{E_1}(mt)]\right)\right|_{t=a}=0,
\end{split}
\end{equation}
and
\begin{equation}
\label{innc-3}
\begin{split}
\left.\left(I_{0+}^\beta [\hat{E_2}(mt)]\right)\right|_{t=a}=0,\,\,
\left.\left(D_{0+}^\gamma [\hat{E_2}(mt)]\right)\right|_{t=a}=1.
\end{split}
\end{equation}
\end{itemize}
\end{thm}

\subsection{Partial differential equations with the non-local Cauchy type data}

Let us consider the time--fractional wave equation
\begin{equation}\label{C2: fde}
\left(D_{0+}^\alpha u\right)(t) + \mathcal A u(t)=f(t),\,\,\, 0<t<T,
\end{equation}
with non-local initial conditions
\begin{equation}\label{C2: initc}
\left.\left(I_{0+}^\beta u\right)(t)\right|_{t=a}=u_{1}, \,\, \left.\left(D_{0+}^\gamma u\right)(t)\right|_{t=a}=u_{2},
\end{equation}
where $1<\alpha\leq 2$, $\mathcal A$ is a self--adjoint operator on the separable Hilbert space $\mathbb H$ satisfying Assumption \ref{Assumption_Sp-Oper}, $u_1$ and $u_2$ are the initial functions.

\begin{thm}
Let $1<\alpha\leq2$, $0<\gamma \leq\alpha-1$ and $0\leq\beta\le 2-\alpha$. Also, we assume that the condition \eqref{inncon} holds for $m=-m_\xi$ for all $\xi\in\mathbb N$. Assume that $f\in L^{2}(0, T; \, \mathbb H)$. Suppose that $u_1, u_2\in\mathbb H$. Then the non-local Cauchy type problem for the time-fractional wave equation \eqref{C2: fde}--\eqref{C2: initc} has a unique solution $u\in L_{\chi}^{2}(0, T; \, \mathbb H)$ such that $\mathcal A u, D_{0+}^\alpha u \in L_{\chi}^{2}(0, T; \, \mathbb H)$ in the form
\begin{equation*}
\begin{split}
u(t)=&\sum\limits_{\xi\in\mathbb N}u_{1\xi}\hat{E_1}(-m_{\xi}t)e_{\xi} + \sum\limits_{\xi\in\mathbb N}u_{2\xi}\hat{E_2}(-m_{\xi}t)e_{\xi}
\\
&+\sum\limits_{\xi\in\mathbb N}\left[\int\limits_0^t s^{\alpha-1}f_{\xi}(t-s)E_{\alpha,\alpha}(-m_{\xi}s^\alpha) ds\right]e_{\xi},
\end{split}
\end{equation*}
where
$$
f_{\xi}=(f, e_\xi)_{\mathbb H}, \,\,\,u_{k\xi}=(u_k, e_\xi)_{\mathbb H}
$$
for $k=1, 2$, for all $\xi\in\mathbb N$.
\end{thm}
\begin{proof}
The proof follows by repeating step by step the proof of Part \textbf{(A)} of Theorem \ref{Th_appl_1}.
\end{proof}

\section{Inner-boundary value problem}

In this Section we find a solution of (\ref{fde}) satisfying the following conditions
\begin{equation}
\label{ibc}
\begin{split}
\left.\left(I_{0+}^\beta u\right)(t)\right|_{t=a}&=\hat{e_1},\,\,0\leq a\leq T, \,\, 0\leq\beta\leq2-\alpha,
\\
\left.\left(D_{0+}^{\gamma} u\right)(t)\right|_{t=b}&=\hat{e_2},\,\,0\leq b\leq T,\,\,0<\gamma\leq \alpha-1.
\end{split}
\end{equation}
Here $\hat{e_1}, \hat{e_2}$ are given real numbers.

Depending on the relation between $a$ and $b$, we have three different inner-boundary value problems. Namely, we have cases $a=b$, $a>b$ and $a<b$. We treat all these problems in a similar way.

Using (\ref{fic1}) and (\ref{fic2}), we find
\begin{equation}
\label{fibc1}
\begin{split}
\left.\left(I_{0+}^{\beta} u\right)(t)\right|_{t=a}&=C_1a^{\alpha+\beta-1}E_{\alpha, \alpha+\beta}(ma^\alpha)+C_2a^{\alpha+\beta-2}E_{\alpha, \alpha+\beta-1}(ma^\alpha)
\\
&+\int\limits_0^af(a-s)s^{\alpha+\beta-1}E_{\alpha,\alpha+\beta}(ms^\alpha)ds=\hat{e_1},
\end{split}
\end{equation}
\begin{equation}
\label{fibc2}
\begin{split}
\left.\left(D_{0+}^{\gamma} u\right)(t)\right|_{t=b}&=C_1b^{\alpha-\gamma-1}E_{\alpha,\alpha-\gamma}(mb^\alpha)+C_2a^{\alpha-\gamma-2}E_{\alpha,\alpha-1-\gamma}(mb^\alpha)
\\
&+\int\limits_0^bf(b-s)s^{\alpha-\gamma-1}E_{\alpha,\alpha-\gamma}(ms^\alpha)dt=\hat{e_2}.
\end{split}
\end{equation}
Equalities (\ref{fibc1})-(\ref{fibc2}) form a system of algebraic equations with respect to unknown constants $C_1$ and $C_2$. This system has a solution if
\begin{equation}
\label{ibcon}
E_{\alpha,\alpha+\beta}(ma^\alpha)E_{\alpha,\alpha-1-\gamma}(mb^\alpha)\neq E_{\alpha,\alpha+\beta-1}(ma^\alpha)E_{\alpha,\alpha-\gamma}(mb^\alpha).
\end{equation}
In this case, a solution of the problem (\ref{fde}), (\ref{ibc}) can be written as
\begin{equation}
\label{ibs}
\begin{split}
u(t)&=\hat e_1\hat{F_1}(mt)+\hat e_2\hat{F_2}(mt)-\int\limits_0^af(a-s)s^{\alpha+\beta-1} E_{\alpha,\alpha+\beta}(ms^\alpha)\hat{F_1}(mt)ds
\\
&-\int\limits_0^bf(b-s)s^{\alpha-1-\gamma}E_{\alpha,\alpha-\gamma}(ms^\alpha)\hat{F_2}(mt)ds
+\int\limits_0^tf(t-s)s^{\alpha-1}E_{\alpha,\alpha}(ms^\alpha)ds,
\end{split}
\end{equation}
where
\begin{equation}
\label{eve2}
\begin{split}
\hat{F_1}(mt)&=\dfrac{t^{\alpha-2}\left[tE_{\alpha,\alpha-1-\gamma}(mb^\alpha)E_{\alpha,\alpha}(mt^\alpha)-bE_{\alpha,\alpha-\gamma}(mb^\alpha)E_{\alpha,\alpha-1}(mt^\alpha)\right]}{a^{\alpha+\beta-2}\left[E_{\alpha,\alpha+\beta}(ma^\alpha)E_{\alpha,\alpha-1-\gamma}(mb^\alpha)- E_{\alpha,\alpha+\beta-1}(ma^\alpha)E_{\alpha,\alpha-\gamma}(mb^\alpha)\right]},
\\
\hat{F_2}(mt)&=\dfrac{t^{\alpha-2}\left[-tE_{\alpha,\alpha-1+\beta}(ma^\alpha)E_{\alpha,\alpha}(mt^\alpha)+aE_{\alpha,\alpha+\beta}(ma^\alpha)E_{\alpha,\alpha-1}(mt^\alpha)\right]}{b^{\alpha-\gamma-2}\left[E_{\alpha,\alpha+\beta}(ma^\alpha)E_{\alpha,\alpha-1-\gamma}(mb^\alpha)- E_{\alpha,\alpha+\beta-1}(ma^\alpha)E_{\alpha,\alpha-\gamma}(mb^\alpha)\right]}.
\end{split}
\end{equation}

These observations in combination with Theorem \ref{Th1} give the following result:
\begin{thm}
\label{Th3}
Assume that $1<\alpha\leq 2$ and
$0\leq\beta\le 2-\alpha$. Let $0<\gamma\leq\alpha-1$. Also, assume that the condition \eqref{ibcon} holds.
\begin{itemize}
    \item[\textbf{(A)}] Suppose $f\in C_{-1}^{1}[0,T]$ if $1<\alpha<2$ and $f\in C_{-1}[0,T]$ if $\alpha=2$.
Then the fractional order differential equation \eqref{fde} satisfying the homogeneous non-local initial conditions \eqref{ibc} with $\hat{e_1}=\hat{e_2}=0$ has a unique solution $u_f\in C_{-1}^{2}[0,T]$ of the form \eqref{ibs}, namely, we have
\begin{equation}
\label{ibs-2}
\begin{split}
u_f(t)&=-\int\limits_0^af(a-s)s^{\alpha+\beta-1} E_{\alpha,\alpha+\beta}(ms^\alpha)\hat{F_1}(mt)ds
\\
&-\int\limits_0^bf(b-s)s^{\alpha-1-\gamma}E_{\alpha,\alpha-\gamma}(ms^\alpha)\hat{F_2}(mt)ds
+\int\limits_0^tf(t-s)s^{\alpha-1}E_{\alpha,\alpha}(ms^\alpha)ds. \end{split}
\end{equation}

\item[\textbf{(B)}] Let $f\equiv0$. Then the solution of the homogeneous equation \eqref{fde} satisfying the non-homogeneous non-local initial conditions \eqref{ibcon} has a unique solution $u_0$ of the following form
\begin{equation}\label{ibs-3}
\begin{split}
u_0(t)=\hat e_1\hat{F_1}(mt)+\hat e_2\hat{F_2}(mt).
\end{split}
\end{equation}
Moreover, we have
\begin{equation}
\label{ibcon-2}
\begin{split}
\left.\left(I_{0+}^\beta [\hat{F_1}(mt)]\right)\right|_{t=a}=1,\,\,
\left.\left(D_{0+}^\gamma [\hat{F_1}(mt)]\right)\right|_{t=b}=0,
\end{split}
\end{equation}
and
\begin{equation}
\label{ibcon-3}
\begin{split}
\left.\left(I_{0+}^\beta [\hat{F_2}(mt)]\right)\right|_{t=a}=0,\,\,
\left.\left(D_{0+}^\gamma [\hat{F_2}(mt)]\right)\right|_{t=b}=1.
\end{split}
\end{equation}
\end{itemize}
\end{thm}

\subsection{Inner-boundary value problem for time-fractional partial differential equations}

Let us consider the time--fractional wave equation
\begin{equation}\label{C3: fde}
\left(D_{0+}^\alpha u\right)(t) + \mathcal A u(t)=f(t),\,\,\, 0<t<T,
\end{equation}
with non-local conditions
\begin{equation}\label{C3: initc}
\left.\left(I_{0+}^\beta u\right)(t)\right|_{t=a}=u_{1}, \,\, \left.\left(D_{0+}^\gamma u\right)(t)\right|_{t=b}=u_{2},
\end{equation}
where $1<\alpha\leq 2$, $\mathcal A$ is a self--adjoint operator on the separable Hilbert space $\mathbb H$ satisfying Assumption \ref{Assumption_Sp-Oper}, $u_1$ and $u_2$ are non-local boundary functions.

\begin{thm}
Let $1<\alpha\leq2$, $0<\gamma \leq\alpha-1$ and $0\leq\beta\le 2-\alpha$. Also, we assume that the condition \eqref{inncon} holds for $m=-m_\xi$ for all $\xi\in\mathbb N$. Assume that $f\in L^{2}(0, T; \, \mathbb H)$. Suppose that $u_1, u_2\in\mathbb H$. Then the non-local Cauchy type problem for the time-fractional wave equation \eqref{C3: fde}--\eqref{C3: initc} has a unique solution $u\in L_{\chi}^{2}(0, T; \, \mathbb H)$ such that $\mathcal A u, D_{0+}^\alpha u \in L_{\chi}^{2}(0, T; \, \mathbb H)$ in the form
\begin{equation*}
\begin{split}
u(t)=&\sum\limits_{\xi\in\mathbb N}u_{1\xi}\hat{F_1}(-m_{\xi}t)e_{\xi} + \sum\limits_{\xi\in\mathbb N}u_{2\xi}\hat{F_2}(-m_{\xi}t)e_{\xi}
\\
&+\sum\limits_{\xi\in\mathbb N}\left[\int\limits_0^t s^{\alpha-1}f_{\xi}(t-s)E_{\alpha,\alpha}(-m_{\xi}s^\alpha) ds\right]e_{\xi},
\end{split}
\end{equation*}
where
$$
f_{\xi}=(f, e_\xi)_{\mathbb H}, \,\,\,u_{k\xi}=(u_k, e_\xi)_{\mathbb H}
$$
for $k=1, 2$, for all $\xi\in\mathbb N$.
\end{thm}
\begin{proof}
The proof follows by repeating step by step the proof of Part \textbf{(A)} of Theorem \ref{Th_appl_1}.
\end{proof}

\end{document}